\documentclass[10pt]{article}
\usepackage{amsmath,amscd}
\usepackage{amssymb,latexsym,amsthm}
\usepackage{color}
\usepackage[spanish,english]{babel}
\usepackage{amssymb}
\usepackage{color}
\usepackage{amsmath,amsthm,amscd}
\usepackage[latin1]{inputenc}
\usepackage{lscape}
\usepackage{fancyhdr}
\usepackage{amsfonts}
\usepackage{pb-diagram}

\numberwithin{equation}{section}
\newtheorem{theorem}{Theorem}[section]

\newtheorem{definition}[theorem]{Definition}
\newtheorem{proposition}[theorem]{Proposition}
\newtheorem{lemma}[theorem]{Lemma}

\newtheorem{corollary}[theorem]{Corollary}

\theoremstyle{definition}

\newtheorem{remark}[theorem]{Remark}

\title{\textbf{Ore and Goldie theorems\\ for skew $PBW$ extensions}}
\author{Oswaldo Lezama\\
\texttt{jolezamas@unal.edu.co}
\\Juan Pablo Acosta, Cristian Chaparro, Ingrid Ojeda, César Venegas\footnote{Students of the \textit{Graduate Program in
Mathematics}.}
\\ Seminario de Álgebra Constructiva - SAC$^2$\\ Departamento de Matemáticas\\ Universidad Nacional de
Colombia, Sede Bogot\'a}
\date{}
\begin{document}
\maketitle
\begin{abstract}
\noindent Many rings and algebras arising in quantum mechanics can be interpreted as skew $PBW$
(Poincaré-Birkhoff-Witt) extensions. Indeed, Weyl algebras, enveloping algebras of
finite-dimensional Lie algebras (and its quantization), Artamonov quantum polynomials, diffusion
algebras, Manin algebra of quantum matrices, among many others, are examples of skew $PBW$
extensions. In this paper we extend the classical Ore and Goldie theorems, known for skew
polynomial rings, to this wide class of non-commutative rings. As application, we prove the quantum
version of the Gelfand-Kirillov conjecture for the skew quantum polynomials.

\bigskip


\bigskip

\end{abstract}
\section{Skew $PBW$ extensions}\label{definitionexamplesspbw}

The classical Ore's theorem says that if $R$ is a left Ore domain and $R[x;\sigma,\delta]$ is the
skew polynomial ring over $R$, with $\sigma$ injective, then $R[x;\sigma,\delta]$ is also a left
Ore domain, and hence has left total division ring of fractions (see \cite{Ore} or also
\cite{Cohn1}). In this paper we generalize this result to skew $PBW$ extensions, a wide class of
non-commutative rings introduced in \cite{LezamaGallego}. Skew $PBW$ extensions include many rings
and algebras arising in quantum mechanics such as the classical $PBW$ extensions (see \cite{Bell}),
Weyl algebras, enveloping algebras of finite-dimensional Lie algebras (and its quantization),
Artamonov quantum polynomials (see \cite{Artamonov}, \cite{Artamonov2}), diffusion algebras, Manin
algebra of quantum matrices, among many others. A very long list of remarkable examples of skew
$PBW$ extensions is presented in \cite{lezamareyes1}, where some ring-theoretic properties have
been investigated for this class of rings, for example, the global, Krull, Goldie and
Gelfand-Kirillov dimensions were estimated. In the present paper we are interested in proving Ore
and Goldie theorems for skew $PBW$ extensions, generalizing this way two well known results.

In this section we recall the definition of skew $PBW$ (Poincaré-Birkhoff-Witt) extensions defined
firstly in \cite{LezamaGallego}, and we will review also some elementary properties about the
polynomial interpretation of this kind of non-commutative rings. Two particular subclasses of these
extensions are recalled also.
\begin{definition}\label{gpbwextension}
Let $R$ and $A$ be rings. We say that $A$ is an \textit{skew $PBW$ extension of $R$} $($also called
a $\sigma-PBW$ extension of $R$$)$ if the following conditions hold:
\begin{enumerate}
\item[\rm (i)]$R\subseteq A$.
\item[\rm (ii)]There exist finite elements $x_1,\dots ,x_n\in A$ such $A$ is a left $R$-free module with basis
\begin{center}
${\rm Mon}(A):= \{x^{\alpha}=x_1^{\alpha_1}\cdots
x_n^{\alpha_n}\mid \alpha=(\alpha_1,\dots ,\alpha_n)\in
\mathbb{N}^n\}$.
\end{center}
In this case it says also that \textit{$A$ is a left polynomial ring over $R$} with respect to
$\{x_1,\dots,x_n\}$ and $Mon(A)$ is the set of standard monomials of $A$. Moreover, $x_1^0\cdots
x_n^0:=1\in Mon(A)$.
\item[\rm (iii)]For every $1\leq i\leq n$ and $r\in R-\{0\}$ there exists $c_{i,r}\in R-\{0\}$ such that
\begin{equation}\label{sigmadefinicion1}
x_ir-c_{i,r}x_i\in R.
\end{equation}
\item[\rm (iv)]For every $1\leq i,j\leq n$ there exists $c_{i,j}\in R-\{0\}$ such that
\begin{equation}\label{sigmadefinicion2}
x_jx_i-c_{i,j}x_ix_j\in R+Rx_1+\cdots +Rx_n.
\end{equation}
Under these conditions we will write $A:=\sigma(R)\langle
x_1,\dots ,x_n\rangle$.
\end{enumerate}
\end{definition}
The following proposition justifies the notation and the
alternative name given for the skew $PBW$ extensions.
\begin{proposition}\label{sigmadefinition}
Let $A$ be an skew $PBW$ extension of $R$. Then, for every $1\leq i\leq n$, there exists an
injective ring endomorphism $\sigma_i:R\rightarrow R$ and a $\sigma_i$-derivation
$\delta_i:R\rightarrow R$ such that
\begin{center}
$x_ir=\sigma_i(r)x_i+\delta_i(r)$,
\end{center}
for each $r\in R$.
\end{proposition}
\begin{proof}
See \cite{LezamaGallego}, Proposition 3.
\end{proof}

A particular case of skew $PBW$ extension is when all derivations $\delta_i$ are zero. Another
interesting case is when all $\sigma_i$ are bijective and the constants $c_{ij}$ are invertible. We
recall the following definition (cf. \cite{LezamaGallego}).
\begin{definition}\label{sigmapbwderivationtype}
Let $A$ be an skew $PBW$ extension.
\begin{enumerate}
\item[\rm (a)]
$A$ is quasi-commutative if the conditions {\rm(}iii{\rm)} and
{\rm(}iv{\rm)} in Definition \ref{gpbwextension} are replaced by
\begin{enumerate}
\item[\rm (iii')]For every $1\leq i\leq n$ and $r\in R-\{0\}$ there exists $c_{i,r}\in R-\{0\}$ such that
\begin{equation}
x_ir=c_{i,r}x_i.
\end{equation}
\item[\rm (iv')]For every $1\leq i,j\leq n$ there exists $c_{i,j}\in R-\{0\}$ such that
\begin{equation}
x_jx_i=c_{i,j}x_ix_j.
\end{equation}
\end{enumerate}
\item[\rm (b)]$A$ is bijective if $\sigma_i$ is bijective for
every $1\leq i\leq n$ and $c_{i,j}$ is invertible for any $1\leq
i<j\leq n$.
\end{enumerate}
\end{definition}

Some extra notation will be used in the paper.

\begin{definition}\label{1.1.6}
Let $A$ be an skew $PBW$ extension of $R$ with endomorphisms $\sigma_i$, $1\leq i\leq n$, as in
Proposition \ref{sigmadefinition}.
\begin{enumerate}
\item[\rm (i)]For $\alpha=(\alpha_1,\dots,\alpha_n)\in \mathbb{N}^n$,
$\sigma^{\alpha}:=\sigma_1^{\alpha_1}\cdots \sigma_n^{\alpha_n}$,
$|\alpha|:=\alpha_1+\cdots+\alpha_n$. If
$\beta=(\beta_1,\dots,\beta_n)\in \mathbb{N}^n$, then
$\alpha+\beta:=(\alpha_1+\beta_1,\dots,\alpha_n+\beta_n)$.
\item[\rm (ii)]For $X=x^{\alpha}\in {\rm Mon}(A)$,
$\exp(X):=\alpha$ and $\deg(X):=|\alpha|$.
\item[\rm (iii)]If $f=c_1X_1+\cdots +c_tX_t$,
with $X_i\in Mon(A)$ and $c_i\in R-\{0\}$, then
$\deg(f):=\max\{\deg(X_i)\}_{i=1}^t.$
\end{enumerate}
\end{definition}
The skew $PBW$ extensions can be characterized in a similar way as
was done in \cite{Gomez-Torrecillas} for $PBW$ rings.
\begin{theorem}\label{coefficientes}
Let $A$ be a left polynomial ring over $R$ w.r.t. $\{x_1,\dots,x_n\}$. $A$ is an skew $PBW$
extension of $R$ if and only if the following conditions hold:
\begin{enumerate}
\item[\rm (a)]For every $x^{\alpha}\in {\rm Mon}(A)$ and every $0\neq
r\in R$ there exist unique elements
$r_{\alpha}:=\sigma^{\alpha}(r)\in R-\{0\}$ and $p_{\alpha ,r}\in
A$ such that
\begin{equation}\label{611}
x^{\alpha}r=r_{\alpha}x^{\alpha}+p_{\alpha , r},
\end{equation}
where $p_{\alpha ,r}=0$ or $\deg(p_{\alpha ,r})<|\alpha|$ if
$p_{\alpha , r}\neq 0$. Moreover, if $r$ is left invertible, then
$r_\alpha$ is left invertible.

\item[\rm (b)]For every $x^{\alpha},x^{\beta}\in {\rm Mon}(A)$ there
exist unique elements $c_{\alpha,\beta}\in R$ and
$p_{\alpha,\beta}\in A$ such that
\begin{equation}\label{612}
x^{\alpha}x^{\beta}=c_{\alpha,\beta}x^{\alpha+\beta}+p_{\alpha,\beta},
\end{equation}
where $c_{\alpha,\beta}$ is left invertible, $p_{\alpha,\beta}=0$
or $\deg(p_{\alpha,\beta})<|\alpha+\beta|$ if
$p_{\alpha,\beta}\neq 0$.
\end{enumerate}
\begin{proof}
See \cite{LezamaGallego}, Theorem 7.
\end{proof}
\end{theorem}
We remember also the following facts from \cite{LezamaGallego}.
\begin{remark}\label{identities}
(i) We observe that if $A$ is quasi-commutative, then $p_{\alpha,r}=0$ and $p_{\alpha,\beta}=0$ for
every $0\neq r\in R$ and every $\alpha,\beta \in \mathbb{N}^n$.

(ii) If $A$ is bijective, then $c_{\alpha,\beta}$ is invertible for any $\alpha,\beta\in
\mathbb{N}^n$.

(iii) In $Mon(A)$ we define
\begin{center}
$x^{\alpha}\succeq x^{\beta}\Longleftrightarrow
\begin{cases}
x^{\alpha}=x^{\beta}\\
\text{or} & \\
x^{\alpha}\neq x^{\beta}\, \text{but} \, |\alpha|> |\beta| & \\
\text{or} & \\
x^{\alpha}\neq x^{\beta},|\alpha|=|\beta|\, \text{but $\exists$
$i$ with} &
\alpha_1=\beta_1,\dots,\alpha_{i-1}=\beta_{i-1},\alpha_i>\beta_i.
\end{cases}$
\end{center}
It is clear that this is a total order on $Mon(A)$. If $x^{\alpha}\succeq x^{\beta}$ but
$x^{\alpha}\neq x^{\beta}$, we write $x^{\alpha}\succ x^{\beta}$. Each element $f\in A$ can be
represented in a unique way as $f=c_1x^{\alpha_1}+\cdots +c_tx^{\alpha_t}$, with $c_i\in R-\{0\}$,
$1\leq i\leq t$, and $x^{\alpha_1}\succ \cdots \succ x^{\alpha_t}$. We say that $x^{\alpha_1}$ is
the \textit{leader monomial} of $f$ and we write $lm(f):=x^{\alpha_1}$ ; $c_1$ is the
\textit{leader coefficient} of $f$, $lc(f):=c_1$, and $c_1x^{\alpha_1}$ is the \textit{leader term}
of $f$ denoted by $lt(f):=c_1x^{\alpha_1}$.
\end{remark}
A natural and useful result that we will use later is the
following property.
\begin{proposition}\label{1.1.10a}
Let A be an skew PBW extension of a ring R. If R is a domain, then A is a domain.
\end{proposition}
\begin{proof}
See \cite{lezamareyes1}.
\end{proof}

The next theorem characterizes the quasi-commutative skew $PBW$ extensions.

\begin{theorem}\label{1.3.3}
Let $A$ be a quasi-commutative skew $PBW$ extension of a ring $R$. Then,
\begin{enumerate}
\item[\rm (i)] $A$ is isomorphic to an iterated skew polynomial ring of
endomorphism type, i.e.,
\begin{center}
$A\cong R[z_1;\theta_1]\cdots [z_{n};\theta_n]$.
\end{center}
\item[\rm (ii)] If $A$ is bijective, then each
endomorphism $\theta_i$ is bijective, $1\leq i\leq n$.
\end{enumerate}
\end{theorem}
\begin{proof}
 See \cite{lezamareyes1}.
\end{proof}

\begin{theorem}\label{1.3.2}
Let $A$ be an arbitrary skew $PBW$ extension of $R$. Then, $A$ is a filtered ring with filtration
given by
\begin{equation}\label{eq1.3.1a}
F_m:=\begin{cases} R & {\rm if}\ \ m=0\\ \{f\in A\mid {\rm deg}(f)\le m\} & {\rm if}\ \ m\ge 1
\end{cases}
\end{equation}
and the corresponding graded ring $Gr(A)$ is a quasi-commutative skew $PBW$ extension of $R$.
Moreover, if $A$ is bijective, then $Gr(A)$ is a quasi-commutative bijective skew $PBW$ extension
of $R$.
\end{theorem}
\begin{proof}
See \cite{lezamareyes1}.
\end{proof}

\begin{theorem}[Hilbert Basis Theorem]\label{1.3.4}
Let $A$ be a bijective skew $PBW$ extension of $R$. If $R$ is a left $($right$)$ Noetherian ring
then $A$ is also a left $($right$)$ Noetherian ring.
\end{theorem}
\begin{proof}
 \cite{lezamareyes1}.
\end{proof}

\section{Preliminary lemmas}

Let us recall first the non-commutative localization. If $R$ is a ring and $S$ is a multiplicative
subset of $R$ (i.e., $1\in S, 0\notin S, ss'\in S$, for $s,s'\in S$) then the left ring of
fractions of $R$ exists if and only if two conditions hold: (i) given $a\in R$ and $s\in S$ such
that $as=0$, then there exists $s'\in S$ such that $s'a=0$; (ii) (\textit{the left Ore condition})
given $a\in R$ and $s\in S$ there exist $s'\in S$ and $a'\in R$ such that $s'a=a's$. When these
conditions hold, the left ring of fractions of $R$ with respect to $S$ is denoted by $S^{-1}R$, and
its elements are classes represented by fractions: two elements $\frac{a}{s}$, $\frac{b}{t}$ are
equal if and only if there exist $c,d\in R$ such that $ca=db, cs=dt\in S$. The operations of
$S^{-1}R$ are given by $\frac{a}{s}+\frac{b}{t}:=\frac{ca+db}{u}$, where $u:=cs=dt\in S$, for some
$c,d\in R$ (the Ore's condition applied to $s$ and $t$), and
$\frac{a}{s}\frac{b}{t}:=\frac{cb}{us}$, where $ua=ct$, for some $u\in S$ and $c\in R$ (the Ore's
condition applied to $a$ and $t$). In a similar way are defined the right rings of fractions. Note
that any domain $R$ satisfies (i) with respect to any multiplicative subset $S$, and it is said
that $R$ is a \textit{left Ore domain} if $R$ satisfies (ii) with respect to $S:=R-\{0\}$. The
elements of the ring $R$ that are non-zero divisors are called \textit{regular} and the set of
regular elements of $R$ will denoted by $S_0(R)$.

In this second section we localize skew polynomial rings and skew $PBW$ extensions by
multiplicative subsets of the ring of coefficients. The basic results presented here will used in
the other sections of the present paper. We start recalling a couple of well known facts.

\begin{proposition}\label{1.3.1a}
Let $\sigma$ be an automorphism of $R$ and $R[x;\sigma,\delta]$ the left skew polynomial ring.
Then, the right skew polynomial ring $R[x;\sigma^{-1},-\delta \sigma^{-1}]_r$ is isomorphic to
$R[x;\sigma,\delta]$.
\end{proposition}
\begin{proof}
See \cite{McConnell}.
\end{proof}

\begin{proposition}\label{2.4.2a}
Let $R$ be a ring and $S\subset R$ a multiplicative subset. If $Q:=S^{-1}R$ exists, then any finite
set $\{q_{1},\ldots, q_{n}\}$ of elements of $Q$ posses a common denominator, i.e., there exist
$r_{1},\ldots,r_{n}\in R$ and $s\in S$ such that $q_{i}=\frac{r_i}{s}, 1\leq i\leq n$.
\end{proposition}
\begin{proof}
See \cite{McConnell}, Lemma 2.1.8.
\end{proof}

The first preliminary result is the following lemma, the first part of which is well known and can
be found in \cite{Gomez-Torrecillas}.

\begin{lemma}\label{2.6.3}
Let $R$ be a ring and $S\subset R$ a multiplicative subset.
\begin{enumerate}
\item[\rm (a)]If $S^{-1}R$ exists and $\sigma(S)\subseteq S$, then
\begin{equation}\label{ecu2.6.1b}
S^{-1}(R[x;\sigma,\delta])\cong (S^{-1}R)[x;\overline{\sigma},\overline{\delta}],
\end{equation}
with
\begin{align*}
S^{-1}R & \xrightarrow{\overline{\sigma}} S^{-1}R & S^{-1}R & \xrightarrow{\overline{\delta}} S^{-1}R\\
\frac{a}{s} & \mapsto \frac{\sigma(a)}{\sigma(s)} & \frac{a}{s} & \mapsto
-\frac{\delta(s)}{\sigma(s)}\frac{a}{s}+\frac{\delta(a)}{\sigma(s)}
\end{align*}
\item[\rm (b)]If $RS^{-1}$ exists and $\sigma$ is bijective with $\sigma(S)=S$, then
\begin{equation}\label{ecu2.6.1c}
(R[x;\sigma,\delta])S^{-1}\cong (RS^{-1})[x;\widetilde{\sigma},\widetilde{\delta}],
\end{equation}
with
\begin{align*}
RS^{-1} & \xrightarrow{\widetilde{\sigma}} RS^{-1} & RS^{-1} & \xrightarrow{\widetilde{\delta}} RS^{-1}\\
\frac{a}{s} & \mapsto \frac{\sigma(a)}{\sigma(s)} & \frac{a}{s} & \mapsto
-\frac{\sigma(a)}{\sigma(s)}\frac{\delta(s)}{s}+\frac{\delta(a)}{s}
\end{align*}
\end{enumerate}
\end{lemma}
\begin{proof}
(a) The sketch of the proof can be found in \cite{Gomez-Torrecillas}, Chapter 8, Lemma 1.10 and
Proposition 1.11.

(b) From Proposition \ref{1.3.1a}, we have $R[x;\sigma^{-1},-\delta \sigma^{-1}]_d\cong
R[x;\sigma,\delta]$. Let $\theta:=\sigma^{-1}$ and $\gamma:=-\delta \sigma^{-1}$, then
$R[x;\theta,\gamma]_d\cong R[x;\sigma,\delta]$, so $(R[x;\theta,\gamma]_d)S^{-1}\cong
(R[x;\sigma,\delta])S^{-1}$. Adapting the proof of \cite{Gomez-Torrecillas}, but for the right side
(the inclusion $\theta(S)\subset S$ is guaranteed by the condition $\sigma(S)=S$), we obtain
\begin{center}
$(R[x;\theta,\gamma]_d)S^{-1}\cong (RS^{-1})[x;\widetilde{\theta},\widetilde{\gamma}]_d$, with
$\widetilde{\theta}(\frac{a}{s}):=\frac{\theta(a)}{\theta(s)}$,
$\widetilde{\gamma}(\frac{a}{s}):=-\frac{a}{s}\frac{\gamma(s)}{\theta(s)}+\frac{\gamma(a)}{\theta(s)}$.
\end{center}
Hence, $(R[x;\sigma,\delta])S^{-1}\cong
(RS^{-1})[x;(\widetilde{\theta})^{-1},-\widetilde{\gamma}(\widetilde{\theta})^{-1}]$. But note that
$(\widetilde{\theta})^{-1}=\widetilde{\sigma}$ and
$-\widetilde{\gamma}(\widetilde{\theta})^{-1}=\widetilde{\delta}$, where
$\widetilde{\sigma},\widetilde{\delta}$ are defined as in the statement of the theorem. In fact, if
$(\widetilde{\theta})^{-1}(\frac{a}{s})=\frac{b}{t}$, then
$\frac{a}{s}=\frac{\theta(b)}{\theta(t)}=\frac{\sigma^{-1}(b)}{\sigma^{-1}(t)}$ and there exist
$c,d\in A$ such that $ac=\sigma^{-1}(b)d$ and $sc=\sigma^{-1}(t)d\in S$. From this we get
$\sigma(a)\sigma(c)=b\sigma(d)$ and $\sigma(s)\sigma(c)=t\sigma(d)\in S$, i.e.,
$\frac{\sigma(a)}{\sigma(s)}=\frac{b}{t}$. For the other equality we have
$-\widetilde{\gamma}(\widetilde{\theta})^{-1}(\frac{a}{s})=-\widetilde{\gamma}(\frac{\sigma(a)}{\sigma(s)})=-[-\frac{\sigma(a)}{\sigma(s)}
\frac{\gamma(\sigma(s))}{\theta(\sigma(s))}+\frac{\gamma(\sigma(a))}{\theta(\sigma(s))}]=-\frac{\sigma(a)}{\sigma(s)}\frac{\delta(s)}{s}+\frac{\delta(a)}{s}=
\widetilde{\delta}(\frac{a}{s})$.
\end{proof}

The previous lemma can be extended to iterated skew polynomial rings.

\begin{corollary}\label{2.1.4}
Let $R$ be a ring and $A:=R[x_1;\sigma_1,\delta_1]\cdots [x_n;\sigma_n,\delta_n]$ the iterated skew
polynomial ring. Let $S$ be a multiplicative system of $R$.
\begin{enumerate}
\item[\rm (a)]If $S^{-1}R$ exists and
$\sigma_i(S)\subseteq S$ for every $1\leq i\leq n$, then
\begin{center}
$S^{-1}A\cong (S^{-1}R)[x_1;\overline{\sigma_1},\overline{\delta_1}]\cdots
[x_n;\overline{\sigma_n},\overline{\delta_n}]$,
\end{center}
with
\begin{align*}
(S^{-1}R)[x_1;\overline{\sigma_1},\overline{\delta_1}]\cdots
[x_{i-1};\overline{\sigma_{i-1}},\overline{\delta_{i-1}}] & \xrightarrow{\overline{\sigma_i}}
(S^{-1}R)[x_1;\overline{\sigma_1},\overline{\delta_1}]\cdots [x_{i-1};\overline{\sigma_{i-1}},\overline{\delta_{i-1}}]\\
\frac{a}{s} & \mapsto \frac{\sigma_i(a)}{\sigma_i(s)}
\end{align*}
\begin{align*}
(S^{-1}R)[x_1;\overline{\sigma_1},\overline{\delta_1}]\cdots
[x_{i-1};\overline{\sigma_{i-1}},\overline{\delta_{i-1}}] & \xrightarrow{\overline{\delta_i}}
(S^{-1}R)[x_1;\overline{\sigma_1},\overline{\delta_1}]\cdots [x_{i-1};\overline{\sigma_{i-1}},\overline{\delta_{i-1}}]\\
\frac{a}{s} & \mapsto -\frac{\delta_i(s)}{\sigma_i(s)}\frac{a}{s}+\frac{\delta_i(a)}{\sigma_i(s)}
\end{align*}
\item[\rm (b)]If $RS^{-1}$ exists and $\sigma_i$ is bijective with $\sigma_i(S)=S$ for every $1\leq i\leq n$, then
\begin{equation*}
AS^{-1}\cong (RS^{-1})[x_1;\widetilde{\sigma_1},\widetilde{\delta_1}]\cdots
[x_n;\widetilde{\sigma_n},\widetilde{\delta_n}],
\end{equation*}
with
\begin{align*}
(RS^{-1})[x_1;\widetilde{\sigma_1},\widetilde{\delta_1}]\cdots
[\widetilde{x_{i-1}};\widetilde{\sigma_{i-1}},\widetilde{\delta_{i-1}}] &
\xrightarrow{\widetilde{\sigma_i}} (RS^{-1})[x_1;\widetilde{\sigma_1},\widetilde{\delta_1}]\cdots
[\widetilde{x_{i-1}};\widetilde{\sigma_{i-1}},\widetilde{\delta_{i-1}}]\\
\frac{a}{s} & \mapsto \frac{\sigma_i(a)}{\sigma_i(s)}
\end{align*}
\begin{align*}
(RS^{-1})[x_1;\widetilde{\sigma_1},\widetilde{\delta_1}]\cdots
[\widetilde{x_{i-1}};\widetilde{\sigma_{i-1}},\widetilde{\delta_{i-1}}] &
\xrightarrow{\widetilde{\delta_i}} (RS^{-1})[x_1;\widetilde{\sigma_1},\widetilde{\delta_1}]\cdots
[\widetilde{x_{i-1}};\widetilde{\sigma_{i-1}},\widetilde{\delta_{i-1}}]\\
\frac{a}{s} & \mapsto -\frac{\sigma_i(a)}{\sigma_i(s)}\frac{\delta_i(s)}{s}+\frac{\delta_i(a)}{s}
\end{align*}
\end{enumerate}
\end{corollary}
\begin{proof}
The part (a) of the corollary follows from Lemma \ref{2.6.3} by iteration and observing that
\begin{center}
$(S^{-1}R)[x_1;\overline{\sigma_1},\overline{\delta_1}]\cdots
[x_{i-1};\overline{\sigma_{i-1}},\overline{\delta_{i-1}}]\cong
S^{-1}(R[x_1;\sigma_1,\delta_1]\cdots [x_{i-1};\sigma_{i-1},\delta_{i-1}])$,
\end{center}
thus any element of $(S^{-1}R)[x_1;\overline{\sigma_1},\overline{\delta_1}]\cdots
[x_{i-1};\overline{\sigma_{i-1}},\overline{\delta_{i-1}}]$ can be represented as a fraction
$\frac{a}{s}$, with $a\in R[x_1;\sigma_1,\delta_1]\cdots [x_{i-1};\sigma_{i-1},\delta_{i-1}]$ and
$s\in S$. The same remark apply for the part (b).
\end{proof}

\begin{corollary}\label{2.1.5}
Let $A:=R[z_1;\sigma_1]\cdots [z_{n};\sigma_n]$ be a quasi-commutative skew $PBW$ extension of a
ring $R$ and let $S$ be a multiplicative system of $R$.
\begin{enumerate}
\item[\rm (a)]If $S^{-1}R$ exists and
$\sigma_i(S)\subseteq S$ for every $1\leq i\leq n$, then
\begin{center}
$S^{-1}A\cong (S^{-1}R)[z_1;\overline{\sigma_1}]\cdots [z_{n};\overline{\sigma_n}]$.
\end{center}
In particular, if $A$ is bijective with $\sigma_i(S)=S$ for every $i$, then $S^{-1}A$ is a
quasi-commutative bijective skew $PBW$ extension of $S^{-1}R$.
\item[\rm (b)]If $RS^{-1}$ exists and $A$ is bijective with $\sigma_i(S)=S$ for every $1\leq i\leq n$,
then $AS^{-1}$ is a quasi-commutative bijective skew $PBW$ extension of $RS^{-1}$ and
\begin{equation*}
AS^{-1}\cong (RS^{-1})[x_1;\widetilde{\sigma_1}]\cdots [x_n;\widetilde{\sigma_n}].
\end{equation*}
\end{enumerate}
\end{corollary}
\begin{proof}
This is a direct consequence of the previous corollary. Assuming that each $\sigma_i$ is bijective
and $\sigma_i(S)=S$, for each $1\leq i\leq n$, then each $\overline{\sigma_i}$ is bijective. In
fact, if $\frac{\sigma_i(a)}{\sigma_i(s)}=\frac{0}{1}$, then there exist $c,d\in
R[x_1;\sigma_1]\cdots [x_{i-1};\sigma_{i-1}]$ such that $c\sigma_i(a)=0$ and $c\sigma_i(s)=d\in S$.
Since $\sigma_i$ is surjective and $S=\sigma_i(S)$, there exist $c'\in R[x_1;\sigma_1]\cdots
[x_{i-1};\sigma_{i-1}]$ and $d'\in S$ such that $\sigma_i(c')=c$ and $\sigma_i(d')=d$, hence
$\sigma_i(c'a)=0$ and $\sigma_i(c's)=\sigma_i(d')$, but $\sigma_i$ is injective, then $c'a=0$ and
$c's=d'$. This means that $\frac{a}{s}=0$, i.e., $\overline{\sigma_i}$ is injective. It is clear
that $\overline{\sigma_i}$ is surjective. Finally, if the constants $c_{i,j}$ that define $A$ are
invertible (see Definition \ref{sigmapbwderivationtype}), then $\frac{c_{i,j}}{1}\in S^{-1}A$ are
also invertible.

For the part (b) the proof is analogous.
\end{proof}

Now we consider arbitrary bijective skew $PBW$ extensions and $S$ a multiplicative subset of $R$
consisting of regular elements, i.e., $S\subseteq S_0(R)$. The next powerful lemma generalizes
Lemma 14.2.7 of \cite{McConnell}.

\begin{lemma}\label{2.1.6}
Let $R$ be a ring and $A:=\sigma(R)\langle x_1,\dots, x_n\rangle$ a bijective skew $PBW$ extension
of $R$. Let $S\subseteq S_0(R)$ a multiplicative subset of $R$ such that $\sigma_i(S)=S$, for every
$1\leq i\leq n$, where $\sigma_i$ is defined by Proposition \ref{sigmadefinition}.
\begin{enumerate}
\item[\rm (a)]If $S^{-1}R$ exists, then $S^{-1}A$ exists and it is a bijective skew $PBW$
extension of $S^{-1}R$ with
\begin{center}
$S^{-1}A= \sigma(S^{-1}R)\langle x_1',\dots, x_n'\rangle$,
\end{center}
where $x_i':=\frac{x_i}{1}$ and the system of constants of $S^{-1}R$ is given by
$c_{i,j}':=\frac{c_{i,j}}{1}$, $c_{i,\frac{r}{s}}':=\frac{\sigma_i(r)}{\sigma_i(s)}$, $1\leq
i,j\leq n$.
\item[\rm (b)]If $RS^{-1}$ exists, then $AS^{-1}$ exists and it is a bijective skew $PBW$
extension of $RS^{-1}$ with
\begin{center}
$AS^{-1}= \sigma(RS^{-1})\langle x_1'',\dots, x_n''\rangle$,
\end{center}
where $x_i'':=\frac{x_i}{1}$ and the system of constants of $RS^{-1}$ is given by
$c_{i,j}'':=\frac{c_{i,j}}{1}$, $c_{i,\frac{r}{s}}'':=\frac{\sigma_i(r)}{\sigma_i(s)}$, $1\leq
i,j\leq n$.
\end{enumerate}
\end{lemma}
\begin{proof}
We will use the notation given in Definition \ref{1.1.6} and Remark \ref{identities}.

(a) Let $f\in A$ and $s\in S$ such that $fs=0$. This implies that $f=0$ and hence $sf=0$. In fact,
suppose that $f\neq 0$, let $lt(f):=cx^{\alpha}$, $c\in R-\{0\}$ and $x^{\alpha}\in Mon(A)$. Then
$c\sigma^{\alpha}(s)=0$, but since $\sigma^{\alpha}(s)\in S$, then $c=0$, a contradiction. Now, let
again $f\in A$ and $s\in S$, we have to find $u\in S$ and $g\in A$ such that $uf=gs$. If $f=0$ we
take $u=1$ and $g=0$. Let $f\neq 0$ and again $lt(f):=cx^{\alpha}$, then there exists $u_1\in S$
and $r\in R$ such that $u_1c=r\sigma^{\alpha}(s)$. Consider $u_1f-rx^{\alpha}s$; if
$u_1f-rx^{\alpha}s=0$, then the Ore condition is satisfied. Let $u_1f-rx^{\alpha}s\neq 0$, then
$lm(u_1f-rx^{\alpha}s)\prec lm(f)$. By induction on $lm$, there exists $u_2\in S$ and $g'\in A$
such that $u_2(u_1f-rx^{\alpha}s)=g's$. Thus, $uf=gs$, with $u:=u_2u_1$ and $g:=u_2rx^{\alpha}+g'$.
This proves that $S^{-1}A$ exists.

Let $R':=S^{-1}R$ and $A':=S^{-1}A$; from $R\subseteq A$ we get that $R'\hookrightarrow A'$. In
fact, the correspondence $\frac{r}{s}\mapsto \frac{r}{s}$ is an injective ring homomorphism since
if $\frac{r}{s}=\frac{0}{1}$ in $A'$, then $s^{-1}r=0$ and hence $r=0$, i.e.,
$\frac{r}{s}=\frac{0}{1}$ in $R'$. We denote $x_i':=\frac{x_i}{1}\in A'$ for every $1\leq i\leq n$;
since $S$ has not zero divisors and $Mon\{x_1,\dots,x_n\}$ is a left $R$-basis of $A$, then $A'$ is
a free left $R'$-module with basis $Mon\{x_1',\dots,x_n'\}$.

Let $c_{i,j}':=\frac{c_{i,j}}{1}$, then $c_{i,j}'\neq 0$ and $x_j'x_i'-c_{i,j}'x_i'x_j'\in R'+
R'x_i'+\cdots+R'x_n'$, for every $1\leq i,j\leq n$.

The endomorphisms $\sigma_i$ of $R$ and the $\sigma_i$-derivations $\delta_i$ that define $A$
(Proposition \ref{sigmadefinition}) induce endomorphisms $\overline{\sigma_i}$ of $R'$ and
$\overline{\sigma_i}$-derivations $\overline{\delta_i}$ of $R'$ (see Lemma \ref{2.6.3}). Since
$\sigma_i$ is bijective and $\sigma_i(S)=S$ then each $\overline{\sigma_i}$ is bijective (the proof
is similar as in Corollary \ref{2.1.5}). We we claim that
$x_i'\frac{r}{s}=\overline{\sigma_i}(\frac{r}{s})x_i'+\overline{\delta_i}(\frac{r}{s})$. Indeed,
\begin{center}
$\overline{\sigma_i}(\frac{r}{s})x_i'+\overline{\delta_i}(\frac{r}{s})=\frac{\sigma_i(r)x_i}{\sigma_i(s)}+\frac{\delta_i(r)}{\sigma_i(s)}-
\frac{\delta_i(s)}{\sigma_i(s)}\frac{r}{s}$
\end{center}
and
\begin{center}
$x_i'\frac{r}{s}=\frac{x_i}{1}\frac{r}{s}=\frac{c(x)r}{u}$, with $u\in S$, $c(x)\in A$ and
$ux_i=c(x)s$.
\end{center}
So $\deg(c(x))=1$ and $c(x)$ involves only $x_i$, hence $c(x)=c_1x_i+c_0$, where $c_0,c_1\in R$.
From this we get the relations
\begin{center}
$u=c_1\sigma_i(s)$, $c_1\delta_i(s)+c_0s=0$.
\end{center}
Therefore,
\begin{center}
$x_i'\frac{r}{s}=\frac{c_1x_ir+c_0r}{u}=\frac{c_1\sigma_i(r)x_i}{c_1\sigma_i(s)}+\frac{c_1\delta_i(r)}{c_1\sigma_i(s)}+\frac{c_0r}{c_1\sigma_i(s)}$,
\end{center}
but $\frac{c_1\sigma_i(r)x_i}{c_1\sigma_i(s)}=\frac{\sigma_i(r)x_i}{\sigma_i(s)}$,
$\frac{c_1\delta_i(r)}{c_1\sigma_i(s)}=\frac{\delta_i(r)}{\sigma_i(s)}$ and
$-\frac{\delta_i(s)}{\sigma_i(s)}\frac{r}{s}=-\frac{c_1\delta_i(s)}{c_1\sigma_i(s)}\frac{r}{s}=-\frac{-c_0s}{c_1\sigma_i(s)}
\frac{r}{s}=\frac{c_0r}{c_1\sigma_i(s)}$. This proved the claimed. Thus, given $\frac{r}{s}\in
R'-\{0\}$ there exists $c_{i,\frac{r}{s}}':=\overline{\sigma_i}(\frac{r}{s})\in R'-\{0\}$ such that
$x_i'\frac{r}{s}-c_{i,\frac{r}{s}}'x_i'\in R'$. This completes the proof that $S^{-1}A$ is an skew
$PBW$ extension of $S^{-1}R$.

(b) Let $f\in A$ and $s\in S$ such that $sf=0$, then $f=0$ and $fs=0$. This proved the first
condition for the existence of $AS^{-1}$. Now, we have to find $u\in S$ and $g\in A$ such that
$fu=sg$. If $f=0$ we take $u:=1$ and $g:=0$. Let $f\neq 0$, $lt(f):=cx^{\alpha}$; there exist
$u_1\in S$ and $r\in R$ such that $cu_1=sr$. Consider $f\sigma^{-\alpha}(u_1)-srx^{\alpha}$; if
$f\sigma^{-\alpha}(u_1)=srx^{\alpha}$, then the Ore condition is satisfied. Let
$f\sigma^{-\alpha}(u_1)-srx^{\alpha}\neq 0$, then $lm(f\sigma^{-\alpha}(u_1)-srx^{\alpha})\prec
lm(f)$. By induction on $lm$, there exist $u_2\in S$ and $g'\in A$ such that
$(f\sigma^{-\alpha}(u_1)-srx^{\alpha})u_2=sg'$. Then $fu=sg$, with $u:=\sigma^{-\alpha}(u_1)u_2$
and $g:=rx^{\alpha}u_2+g'$. This proves that $AS^{-1}$ exists.

Let $R'':=RS^{-1}$ and $A'':=AS^{-1}$; from $R\subseteq A$ we get that $R''\hookrightarrow A''$. In
fact, the correspondence $\frac{r}{s}\mapsto \frac{r}{s}$ is an injective ring homomorphism since
if $\frac{r}{s}=\frac{0}{1}$ in $A''$, then $rs^{-1}=0$ and hence $r=0$, i.e.,
$\frac{r}{s}=\frac{0}{1}$ in $R''$.

We note $x_i'':=\frac{x_i}{1}\in A''$ for every $1\leq i\leq n$. Let
$c_{i,j}'':=\frac{c_{i,j}}{1}$, then $c_{i,j}''\neq 0$ and $x_j''x_i''-c_{i,j}''x_i''x_j''\in R''+
R''x_i''+\cdots+R''x_n''$ for every $1\leq i,j\leq n$.

The endomorphisms $\sigma_i$ of $R$ and the $\sigma_i$-derivations $\delta_i$ that define $A$ (see
Proposition \ref{sigmadefinition}) induce endomorphisms $\widetilde{\sigma_i}$ of $R''$ and
$\widetilde{\sigma_i}$-derivations $\widetilde{\delta_i}$ of $R''$ (see Lemma \ref{2.6.3}). Note
that each $\widetilde{\sigma_i}$ is bijective. We claim that
$x_i''\frac{r}{s}=\widetilde{\sigma_i}(\frac{r}{s})x_i''+\widetilde{\delta_i}(\frac{r}{s})$.
Indeed,
\begin{center}
$x_i''\frac{r}{s}=\frac{x_i}{1}\frac{r}{s}=\frac{x_ir}{s}=\frac{\sigma_i(r)x_i+\delta_i(r)}{s}=\frac{\sigma_i(r)x_i}{s}+\frac{\delta_i(r)}{s}$.
\end{center}
On the other hand,
\begin{center}
$\widetilde{\sigma_i}(\frac{r}{s})x_i''+\widetilde{\delta_i}(\frac{r}{s})=\frac{\sigma_i(r)}{\sigma_i(s)}\frac{x_i}{1}-\frac{\sigma_i(r)}{\sigma_i(s)}
\frac{\delta_i(s)}{s}+\frac{\delta_i(r)}{s}$.
\end{center}
Thus, we must prove that
\begin{center}
$\frac{\sigma_i(r)x_i}{s}=\frac{\sigma_i(r)}{\sigma_i(s)}\frac{x_i}{1}-\frac{\sigma_i(r)}{\sigma_i(s)}
\frac{\delta_i(s)}{s}$
\end{center}
Applying the right Ore condition to $\sigma_i(s)$ and $x_i$ we get that $x_iu=\sigma_i(s)c(x)$,
with $u\in S$ and $c(x)\in A$. As in the part (a), $c(x)=cx_i+d$, with $c,d\in R$, so
$\sigma_i(u)=\sigma_i(s)c$ and $\delta_i(u)=\sigma_i(s)d$. Thus,
\begin{center}
$\frac{\sigma_i(r)}{\sigma_i(s)}\frac{x_i}{1}=\frac{\sigma_i(r)(cx_i+d)}{u}=\frac{\sigma_i(r)cx_i}{u}+\frac{\sigma_i(r)d}{u}$
\end{center}
and hence
\begin{center}
$\frac{\sigma_i(r)}{\sigma_i(s)}\frac{x_i}{1}-\frac{\sigma_i(r)}{\sigma_i(s)}
\frac{\delta_i(s)}{s}=\frac{\sigma_i(r)cx_i}{u}+\frac{\sigma_i(r)d}{u}-\frac{\sigma_i(r)}{\sigma_i(s)}
\frac{\delta_i(s)}{s}$,
\end{center}
but $u=s\sigma_i^{-1}(c)$, so
\begin{center}
$\frac{\sigma_i(r)cx_i}{u}=\frac{\sigma_i(r)}{1}\frac{cx_i}{u}=\frac{\sigma_i(r)}{1}\frac{x_i\sigma_i^{-1}(c)-\delta_i(\sigma_i^{-1}(c))}{s\sigma_i^{-1}(c)}
=\frac{\sigma_i(r)}{1}\frac{x_i\sigma_i^{-1}(c)}{s\sigma_i^{-1}(c)}-\frac{\delta_i(\sigma_i^{-1}(c))}{s\sigma_i^{-1}(c)}=
\frac{\sigma_i(r)}{1}\frac{x_i}{s}-\frac{\delta_i(\sigma_i^{-1}(c))}{s\sigma_i^{-1}(c)}=\frac{\sigma_i(r)x_i}{s}-\frac{\delta_i(\sigma_i^{-1}(c))}{s\sigma_i^{-1}(c)}$.
\end{center}
Hence, the problem is reduced to prove the equality
\begin{center}
$\frac{\sigma_i(r)d}{u}-\frac{\sigma_i(r)}{\sigma_i(s)}\frac{\delta_i(s)}{s}=\frac{\delta_i(\sigma_i^{-1}(c))}{s\sigma_i^{-1}(c)}$
\end{center}
or equivalently, to prove
\begin{center}
$\frac{\sigma_i(r)d}{u}-\frac{\delta_i(\sigma_i^{-1}(c))}{u}=\frac{\sigma_i(r)}{\sigma_i(s)}\frac{\delta_i(s)}{s}$.
\end{center}
Note that
$\delta_i(u)=\sigma_i(s)\delta_i(\sigma_i^{-1}(c))+\delta_i(s)\sigma_i^{-1}(c)=\sigma_i(s)d$, i.e.,
$\delta_i(s)\sigma_i^{-1}(c)=\sigma_i(s)(d-\delta_i(\sigma_i^{-1}(c)))$. But this relation
indicates that
\begin{center}
$\frac{\sigma_i(r)}{\sigma_i(s)}\frac{\delta_i(s)}{s}=\frac{\sigma_i(r)(d-\delta_i(\sigma_i^{-1}(c)))}{s\sigma_i^{-1}(c)}=
\frac{\sigma_i(r)d}{u}-\frac{\delta_i(\sigma_i^{-1}(c))}{u}$.
\end{center}
This proved the claimed. Thus, given $\frac{r}{s}\in R''-\{0\}$ there exists
$c_{i,\frac{r}{s}}'':=\widetilde{\sigma_i}(\frac{r}{s})\in R''-\{0\}$ such that
$x_i''\frac{r}{s}-c_{i,\frac{r}{s}}''x_i''\in R''$.

Now we will show that $A''$ is a free left $R''$-module with basis $Mon\{x_1'',\dots,x_n''\}$.
First note that $A''$ is generated by $Mon\{x_1'',\dots,x_n''\}$. In fact, let $z\in A''$, then $z$
has the form $z=\frac{(c_1x^{\alpha_1}+\cdots+c_tx^{\alpha_t})}{s}$, with $c_i\in R$,
$x^{\alpha_i}\in Mon\{x_1,\dots,x_n\}$, $1\leq i\leq t$, and $s\in S$. It is enough to show that
each summand $\frac{cx^{\alpha}}{s}$ is generated by $Mon\{x_1'',\dots,x_n''\}$. But observe that
$\frac{cx^{\alpha}}{s}=\frac{c}{1}\frac{x^{\alpha}}{1}\frac{1}{s}=\frac{c}{1}x''^{\alpha}\frac{1}{s}$
and, as in the proof of the part (a) of Theorem \ref{coefficientes},
$x''^{\alpha}\frac{1}{s}=(\widetilde{\sigma})^{\alpha}(\frac{1}{s})x''^{\alpha}+p_{\alpha,\frac{1}{s}}$,
where $p_{\alpha,\frac{1}{s}}$ is a left linear combination of elements of
$Mon\{x_1'',\dots,x_n''\}$ with coefficients in $R''$. Thus, $A''$ is left generated over $R''$ by
$Mon\{x_1'',\dots,x_n''\}$. Now let $\frac{r_1}{s_1},\dots,\frac{r_t}{s_t}\in R''$ and
$x''^{\alpha_1},\dots,x''^{\alpha_t}\in Mon\{x_1'',\dots,x_n''\}$ such that
$\frac{r_1}{s_1}x''^{\alpha_1}+\cdots+\frac{r_t}{s_t}x''^{\alpha_t}=0$. Taking common denominator
(Proposition \ref{2.4.2a}), and without lost of generality, we can write
$\frac{r_1}{s}\frac{x^{\alpha_1}}{1}+\cdots+\frac{r_t}{s}\frac{x^{\alpha_t}}{1}=0$, with $s\in S$;
moreover, we can assume that $x^{\alpha_1}\succ x^{\alpha_2}\succ\cdots \succ x^{\alpha_t}$. There
exist $u_i\in S$ and $g_i\in A$ such that $x^{\alpha_i}u_i=sg_i$, $1\leq i\leq t$, so
$\frac{r_1g_1}{u_1}+\cdots+\frac{r_tg_t}{u_t}=0$. Note that every $g_i\neq0$. Applying repeatedly
the Ore condition we find elements $a_i\in R$ such that $u_ia_i=u\in S$ and
$\frac{r_1g_1a_1}{u}+\cdots+\frac{r_tg_ta_t}{u}=0$. From this we find $w\in S$ such that
$r_1g_1a_1w+\cdots+r_tg_ta_tw=0$. Let $g_i=c_ix^{\beta_i}+g_i'$, with $lt(g_i)=c_ix^{\beta_i}\neq
0$ and $g_i'\in A$. From $x^{\alpha_i}u_i=sg_i$ we get that $\sigma^{\alpha_i}(u_i)=sc_i$ and
$\alpha_i=\beta_i$. In particular, $\sigma^{\alpha_1}(u_1)=sc_1$; moreover,
$r_1c_1\sigma^{\beta_1}(a_1w)=0$, but since $w\in S$, then $r_1c_1\sigma^{\beta_1}(a_1)=0$. Thus,
we have $r_1(c_1\sigma^{\beta_1}(a_1))=0$ and
$s(c_1\sigma^{\beta_1}(a_1))=\sigma^{\alpha_1}(u_1)\sigma^{\beta_1}(a_1)=\sigma^{\alpha_1}(u_1a_1)=\sigma^{\alpha_1}(u)\in
S$. This means that $\frac{r_1}{s}=0$. By induction on $t$ we get that every $\frac{r_i}{s}=0$.
This completes the proof that $AS^{-1}$ is an skew $PBW$ extension of $RS^{-1}$.
\end{proof}

\section{Ore's theorem}

This section deals with establishing sufficient conditions for an skew $PBW$ extension $A$ of a
ring $R$ be left (right) Ore domain, and hence, $A$ has left (right) total division ring of
fractions. In particular, we will extend the Ore's theorem to skew $PBW$ extensions. A first
elementary result is the following proposition.

\begin{proposition}\label{1.5.1}
If $R$ is a left {\rm(}right{\rm)} Noetherian domain and $A$ is a bijective skew $PBW$ extension of
$R$, then $A$ is a left {\rm(}right{\rm)} Ore domain, and hence, the left $($right$)$ division ring
of fractions of $A$ exists.
\end{proposition}
\begin{proof}
It is well known that left (right) Noetherian domains are left (right) Ore domains (see
\cite{McConnell}, Theorem 2.1.15). The result is consequence of Proposition \ref{1.1.10a} and
Theorem \ref{1.3.4}.
\end{proof}

The main purpose of the present section is to replace the Noetherianity in Proposition \ref{1.5.1}
by the Ore condition. A preliminary result is needed.

\begin{proposition}\label{2.2.3}
Let $B$ be a domain and $S$ a multiplicative subset of $B$ such that $S^{-1}B$ exists. Then, $B$ is
left Ore domain if and only if $S^{-1}B$ is a left Ore domain. In such case
\begin{center}
$Q_l(B)\cong Q_l(S^{-1}B)$.
\end{center}
The right side version of the proposition holds too.
\end{proposition}
\begin{proof}
(i) $\Rightarrow)$: Note first that $S^{-1}B$ is a domain: let $\frac{a}{s},\frac{b}{t}\in S^{-1}B$
such that $\frac{a}{s}\frac{b}{t}=\frac{0}{1}$. There exist $u\in S$ and $c\in B$ such that $ua=ct$
and $\frac{cb}{us}=\frac{0}{1}$. Hence, there exist $c',d'\in B$ such that $c'cb=0$ and
$c'us=d'1\in S$. Since $B$ is a domain $cb=0$, then $b=0$ or $c=0$, and in this last case we get
that $a=0$. Thus, $\frac{a}{s}=0$ or $\frac{b}{t}=0$.

Let again $\frac{a}{s},\frac{b}{t}\in S^{-1}B$ with $\frac{b}{t}\neq 0$, then $b\neq 0$ and there
exist $p\neq 0$ and $q$ in $B$ such that $pa=qb$. Then,
$\frac{ps}{1}\frac{a}{s}=\frac{pa}{1}=\frac{qb}{1}=\frac{qt}{1}\frac{b}{t}$, with $\frac{ps}{1}\neq
0$.

$\Leftarrow)$: Let $a,u\in B$, $u\neq 0$, then $\frac{a}{1},\frac{u}{1}\in S^{-1}B$, with
$\frac{u}{1}\neq 0$. There exist $\frac{c}{t},\frac{d}{s}\in S^{-1}A$, with $\frac{c}{t}\neq 0$
such that $\frac{c}{t}\frac{a}{1}=\frac{d}{s}\frac{u}{1}$, i.e., $\frac{ca}{t}=\frac{du}{s}$. There
exist $c',d'\in B$ such that $c'ca=d'du$ and $c't=d's\in S$. Note that $c'c\neq 0$ since $c'\neq 0$
and $c\neq 0$.

(ii) The function
\begin{align*}
\varphi:S^{-1}B&\to Q_l(B)\\
\frac{b}{s}&\mapsto \frac{b}{s}
\end{align*}
verify the conditions that define a left total ring of fractions, i.e., $\varphi$ is an injective
ring homomorphism, the non-zero elements of $S^{-1}B$ are invertible in $Q_l(B)$  and each element
$\frac{b}{u}$ of $Q_l(B)$ can be written as
$\frac{b}{u}=\varphi(\frac{u}{1})^{-1}\varphi(\frac{b}{1})$.
\end{proof}

\begin{proposition}\label{Ore's theorem}
If $R$ is a left Ore domain and $\sigma$ is injective, then $R[x;\sigma,\delta]$ is a left Ore
domain and
\begin{equation}\label{ecu2.6.7}
Q_l(R[x;\sigma,\delta])\cong Q_l(Q_l(R)[x;\overline{\sigma},\overline{\delta}]),
\end{equation}
If $R$ is a right Ore domain and $\sigma$ is bijective, then $R[x;\sigma,\delta]$ is a right Ore
domain and
\begin{equation}
Q_d(R[x;\sigma,\delta])\cong Q_d(Q_d(R)[x;\widetilde{\sigma},\widetilde{\delta}]).
\end{equation}
\end{proposition}
\begin{proof}
The conditions in (a) of Lemma \ref{2.6.3} are trivially satisfied for $S:=R-\{0\}$. Thus,
$Q_l(R)[x;\overline{\sigma},\overline{\delta}]$ is a well-defined skew polynomial ring over the
division ring $Q_l(R)$ and we have the isomorphism $S^{-1}(R[x;\sigma,\delta])\cong
Q_l(R)[x;\overline{\sigma},\overline{\delta}]$. Note that $\overline{\sigma}$ is injective, and
hence $Q_l(R)[x;\overline{\sigma},\overline{\delta}]$ is a left Noetherian domain and therefore a
left Ore domain. From this we get that $S^{-1}(R[x;\sigma,\delta])$ is a left Ore domain. From
Proposition \ref{2.2.3}, $R[x;\sigma,\delta]$ is a left Ore domain and
$Q_l(R[x;\sigma,\delta])\cong Q_l(S^{-1}(R[x;\sigma,\delta]))\cong Q_l(Q_l(R)[x;
\overline{\sigma},\overline{\delta}])$. This proves(\ref{ecu2.6.7}).

For the second statement note that if $R$ is a right Ore domain, then the right skew polynomial
ring is a right Ore domain. Therefore, Proposition \ref{1.3.1a} guarantees that if $R$ is a right
Ore domain, then $R[x;\sigma,\delta]$ is a right Ore domain, and from (\ref{ecu2.6.1c}) of Lemma
\ref{2.6.3} we get $Q_d(R[x;\sigma,\delta])\cong
Q_d(Q_d(R)[x;\widetilde{\sigma},\widetilde{\delta}])$.
\end{proof}

\begin{corollary}\label{2.1.8}
Let $R$ be a left Ore domain and $A:=R[x_1;\sigma_1,\delta_1]\cdots [x_n;\sigma_n,\delta_n]$, with
$\sigma_i$ injective for every $1\leq i\leq n$. Then, $A$ is a left Ore domain and
\begin{equation*}
Q_l(A)\cong Q_l(Q_l(R)[x_1;\overline{\sigma_1},\overline{\delta_1}]\cdots
[x_n;\overline{\sigma_n},\overline{\delta_n}]),
\end{equation*}
If $R$ is a right Ore domain and $\sigma_i$ is bijective for every $1\leq i\leq n$, then $A$ is a
right Ore domain and
\begin{equation*}
Q_d(A)\cong Q_d(Q_d(R)[x_1;\widetilde{\sigma},\widetilde{\delta_1}],\cdots,
[x_n;\widetilde{\sigma},\widetilde{\delta_n}]).
\end{equation*}
\end{corollary}
\begin{proof}
The result follows from Proposition \ref{Ore's theorem} by iteration.
\end{proof}

\begin{theorem}[Ore's theorem: quasi-commutative case]\label{2.1.9}
Let $R$ be a left Ore domain and $A:=R[x_1;\sigma_1]\cdots [x_{n};\sigma_n]$ be a quasi-commutative
skew $PBW$ extension of $R$. Then $A$ is a left Ore domain, and hence, $A$ has left total division
ring of fractions such that
\begin{center}
$Q_l(A)\cong Q_l(Q_l(R)[x_1;\overline{\sigma_1}]\cdots [x_n;\overline{\sigma_n}])$.
\end{center}
If $R$ is a right Ore domain and $\sigma_i$ is bijective for every $1\leq i\leq n$, then $A$ is a
right Ore domain and
\begin{equation*}
Q_d(A)\cong Q_d(Q_d(R)[x_1;\widetilde{\sigma}],\cdots, [x_n;\widetilde{\sigma}]).
\end{equation*}
\end{theorem}
\begin{proof}
This follows from Corollary \ref{2.1.8} since for any skew $PBW$ extension the endomorphisms
$\sigma$'s are always injective, see Proposition \ref{sigmadefinition}.
\end{proof}

Now we consider the previous theorem for bijective extensions, extending this way Proposition
\ref{1.5.1} to left (right) Ore domains.

\begin{theorem}[Ore's theorem: bijective case]\label{Orespbw}
Let $A=\sigma(R)\langle x_1,\dots, x_n\rangle$ be a bijective skew $PBW$ extension of a left Ore
domain $R$. Then $A$ is also a left Ore domain, and hence, $A$ has left total division ring of
fractions such that
\begin{center}
$Q_l(A)\cong Q_l(\sigma(Q_l(R))\langle x_1',\dots, x_n'\rangle)$.
\end{center}
If $R$ is a right Ore domain, then $A$ is also a right Ore domain, and hence, $A$ has right total
division ring of fractions such that
\begin{center}
$Q_d(A)\cong Q_d(\sigma(Q_d(R))\langle x_1'',\dots, x_n''\rangle)$.
\end{center}
\begin{proof}
With $S:=R-\{0\}$ in Lemma \ref{2.1.6}, $S^{-1}A=\sigma(Q_l(R))\langle x_1',\dots, x_n'\rangle$ is
a left Ore domain. In fact, we have that $Q_l(R)$ is a division ring, so from Theorem \ref{1.3.4}
and Proposition \ref{1.1.10a} we obtain that $\sigma(Q_l(R))\langle x_1',\dots, x_n'\rangle$ is a
left Noetherian domain, and hence, a left Ore domain. From Proposition \ref{2.2.3} we get that $A$
is a left Ore domain and $Q_l(A)\cong Q_l(S^{-1}A)\cong Q_l(\sigma(Q_l(R))\langle x_1',\dots,
x_n'\rangle)$. The proof for the right side is analogous.
\end{proof}
\end{theorem}

\section{Goldie's theorem}\label{SectionArtinianquotientringskewPBW}

Now we pass to study the second classical theorem that we want to prove for the skew $PBW$
extensions. Goldie's theorem says that a ring $B$ has semisimple left (right) total rings of
fractions if and only if $B$ is semiprime and left (right) Goldie. In particular, $B$ has simple
left (right) Artinian left (right) total ring of fractions if and only if $B$ is prime and left
(right) Goldie (see \cite{Goodearl}). In this section we study this result for skew $PBW$
extensions.

The first remark for this problem is the following proposition.

\begin{proposition}\label{GoldiespbwL}
Let $R$ be a prime left $($right$)$ Noetherian ring and let $A$ be a bijective skew \textit{PBW}
extension of $R$. Then $A$ has left $($right$)$ total ring of fractions $Q_l(A)$ which is simple
and left $($right$)$ Artinian.
\begin{proof}
By Theorem \ref{1.3.4}, we know that $A$ is left $($right$)$ Noetherian and hence left $($right$)$
Goldie. Now, observe that $A$ is also a prime ring. In fact, it is well known that an skew
polynomial ring of automorphism type over a prime ring is prime (\cite{McConnell}, Theorem 1.2.9.),
hence, from Theorems \ref{1.3.3} and \ref{1.3.2} we conclude that $Gr(A)$ is a prime ring, whence,
$A$ is prime (see \cite{McConnell}, Proposition 1.6.6). The assertion of the proposition follows
from Goldie's theorem.
\end{proof}
\end{proposition}

Next we want to extend the previous proposition to the case when the ring $R$ of coefficients is
semiprime and left (right) Goldie. We will consider separately the quasi-commutative and bijective
cases. We start recalling the following recent result that motivated us to investigate Goldie's
theorem for skew $PBW$ extensions.

\begin{proposition}
Let $R$ be a semiprime left Goldie ring and let $\sigma$ be injective. Then, $R[x;\sigma,\delta]$
is semiprime left Goldie, and hence, $Q_l(R[x;\sigma,\delta])$ exists and it is semisimple. If $R$
is right Goldie and $\sigma$ is bijective, then $R[x;\sigma,\delta]$ is semiprime right Goldie, and
hence, $Q_r(R[x;\sigma,\delta])$ exists and it is semisimple.
\end{proposition}
\begin{proof}
See \cite{LeroyMatczuk2005}, Theorem 3.8. For the second part we use also Proposition \ref{1.3.1a}.
\end{proof}

\begin{corollary}\label{2.3.2}
Let $R$ be a semiprime left Goldie ring and $\sigma_i$ injective for every $1\leq i\leq n$. Then,
$A:=R[x_1;\sigma_1,\delta_1]\cdots [x_n;\sigma_n,\delta_n]$ is semiprime left Goldie, and hence,
$Q_l(A)$ exists and it is semisimple. If $R$ is right Goldie and every $\sigma_i$ is bijective,
then $A$ is semiprime right Goldie, and hence, $Q_r(A)$ exists and it is semisimple.
\end{corollary}
\begin{proof}
Direct consequence of the previous proposition by iteration.
\end{proof}

\begin{theorem}[Goldie's theorem: quasi-commutative case]\label{2.3.3}
Let $R$ be a semiprime left Goldie ring and $A$ a quasi-commutative skew $PBW$ extension of $R$.
Then, $A$ is semiprime left Goldie, and hence, $Q_l(A)$ exists and it is semisimple. If $R$ is
right Goldie and every $\sigma_i$ is bijective, then $A$ is semiprime right Goldie, and hence,
$Q_r(A)$ exists and it is semisimple.
\end{theorem}
\begin{proof}
This follows from Theorem \ref{1.3.3} and the previous corollary.
\end{proof}

Next we consider Goldie's theorem for bijective extensions. Some preliminaries are needed. Recall
that an element $x$ of a ring $B$ is \textit{left regular} if $rx=0$ implies that $r=0$ for $r\in
B$. We start considering rings for which the set of left regular elements coincides with the set of
regular elements. One remarkable example of this class of rings are the semiprime left Goldie rings
(see \cite{McConnell}, Proposition 2.3.4). Similar statements are true for the right side.

\begin{proposition}\label{2.3.4a}
Let $B$ be a ring and $S\subseteq S_0(B)$ a multiplicative system of $B$ such that $S^{-1}B$
exists. Suppose that any left regular element of $S^{-1}B$ is regular, then the same holds for $B$.
The right side version of the proposition is also true.
\end{proposition}
\begin{proof}
Let $a\in B$ be a left regular element, and let $b\in B$ such that $ab=0$. Then $\frac{a}{1}$ is a
left regular element of $S^{-1}B$. In fact, if $\frac{c}{u}\frac{a}{1}=0$, then $\frac{ca}{u}=0$,
i.e., $\frac{ca}{1}=0$, but since $S$ has not zero divisors, then $ca=0$. This implies that $c=0$,
i.e., $\frac{c}{u}=0$. Now, from $ab=0$ we get $\frac{a}{1}\frac{b}{1}=0$, and by the hypothesis,
$\frac{b}{1}=0$, i.e., $b=0$.
\end{proof}

\begin{proposition}\label{2.3.4}
Let $B$ be a ring such that any left regular element is regular. Let $S\subseteq S_0(B)$ a
multiplicative system of $B$ such that $S^{-1}B$ exists. Then,
\begin{enumerate}
\item[\rm (i)]$Q_l(B)$ exists if and only if $Q_l(S^{-1}B)$ exists. In such case,
\begin{center}
$Q_l(B)\cong Q_l(S^{-1}B)$.
\end{center}
\item[\rm (ii)]$B$ is semiprime left Goldie if and only if $S^{-1}B$ is semiprime left Goldie.
\end{enumerate}
The right side version of the proposition holds.
\end{proposition}
\begin{proof}
(i) $\Rightarrow)$: Let $\frac{a}{s}\in S^{-1}B$ and $\frac{b}{t}\in S_0(S^{-1}B)$. Note that $b\in
S_0(B)$. In fact, if $bc=0$ for some $c\in B$, then $\frac{b}{t}\frac{c}{1}=0$ and hence
$\frac{c}{1}=0$. Since $S$ has not zero divisors, then $c=0$. On the other hand, if $db=0$ for some
$d\in B$, then $\frac{dt}{1}\frac{b}{t}=\frac{db}{1}=0$. This implies that $\frac{dt}{1}=0$, and
hence, $d=0$.

By the hypothesis, there exist $z\in S_0(B)$ and $z'\in B$ such that $za=z'b$. From this we obtain
$\frac{zs}{1}\frac{a}{s}=\frac{z't}{1}\frac{b}{t}$, but observe that $zs\in S_0(B)$ and hence,
$\frac{zs}{1}\in S_0(S^{-1}B)$. In fact, we will show that if $u\in S_0(B)$, then $\frac{u}{1}\in
S_0(S^{-1}B)$. Let $\frac{p}{v}\in S^{-1}B$ such that $\frac{p}{v}\frac{u}{1}=0$, then
$\frac{pu}{v}=0$, so $\frac{pu}{1}=0$ and hence $p=0$, i.e., $\frac{p}{v}=0$. Now, let
$\frac{q}{w}\in S^{-1}B$ such that $\frac{u}{1}\frac{q}{w}=0$. There exist $v\in S$ and $x\in B$
such that $vu=xw$ and $\frac{xq}{u}=0$, i.e., $xq=0$. Note that $x$ is left regular since $vu$ is
regular, then by the hypothesis $x$ is regular, and hence, $q=0$, i.e., $\frac{q}{w}=0$.

This proves that $Q_l(S^{-1}B)$ exists.

$\Leftarrow)$: Let $a\in B$ and $u\in S_0(B)$, then $\frac{a}{1}, \frac{u}{1}\in S^{-1}B$ and, as
above, $\frac{u}{1}\in S_0(S^{-1}B)$. By the hypothesis, there exist $\frac{z}{s},\frac{z'}{s'}\in
S^{-1}B$ with $\frac{z'}{s'}\in S_0(S^{-1}B)$ such that
$\frac{z'}{s'}\frac{a}{1}=\frac{z}{s}\frac{u}{1}$, i.e., $\frac{z'a}{s'}=\frac{zu}{s}$, so there
exist $c,d\in B$ such that $cz'a=dzu$ and $cs'=ds\in S$. In order to complete the proof of the left
Ore condition we have to show that $cz'\in S_0(B)$. If $xcz'=0$ for some $x\in B $, then
$\frac{xc}{1}\frac{z'}{1}=\frac{0}{1}$, i.e.,
$\frac{xc}{1}\frac{s'}{1}\frac{1}{s'}\frac{z'}{1}=\frac{0}{1}$, so
$\frac{xcs'}{1}\frac{z'}{s'}=\frac{0}{1}$, and hence $\frac{xcs'}{1}=\frac{0}{1}$. This means
$xcs'=0$, so $x=0$. Now, if $cz'p=0$ for some $p\in B$, then
$\frac{c}{1}\frac{s'}{1}\frac{1}{s'}\frac{z'}{1}\frac{p}{1}=0=\frac{cs'}{1}\frac{z'}{s'}\frac{p}{1}$,
but since $cs'\in S$ we get that $\frac{cs'}{1}\in S_0(S^{-1}B)$, and hence
$\frac{z'}{s'}\frac{p}{1}=0$, and from this we obtain $\frac{p}{1}=0$, i.e., $p=0$. This proves
that $Q_l(B)$ exists.

The function
\begin{align*}
\varphi:S^{-1}B&\to Q_l(B)\\
\frac{b}{s}&\mapsto \frac{b}{s}
\end{align*}
verify the four conditions that define a left total ring of fractions, i.e., (a) $\varphi$ is a
ring homomorphism. (b) $S_0(S^{-1}B)\subseteq Q_l(B)^*$: in fact, let $\frac{b}{t}\in
S_0(S^{-1}B)$, then as we observed at the beginning of the proof, $b\in S_0(B)$, and hence,
$\varphi(\frac{b}{t})=\frac{b}{t}$ is invertible in $Q_l(B)$ with inverse $\frac{t}{b}$. (c)
$\frac{b}{s}\in \ker(\varphi)$ if and only if $\frac{d}{1}\frac{b}{s}=0$ with $\frac{d}{1}\in
S_0(S^{-1}B)$: in fact, if $\frac{b}{s}\in \ker(\varphi)$, then there exist $c,d\in B$ such that
$cb=0$ and $cs=d$, with $d\in S_0(B)$, but this means that $\frac{d}{1}\frac{b}{s}=0$, with
$\frac{d}{1}\in S_0(S^{-1}B)$. The converse is trivial. (d) each element $\frac{b}{u}$ of $Q_l(B)$
can be written as $\frac{b}{u}=\varphi(\frac{u}{1})^{-1}\varphi(\frac{b}{1})$.

(ii) This numeral is a direct consequence of (i) and the Goldie's theorem.
\end{proof}

\begin{proposition}\label{2.3.5}
Let $B$ be a positive filtered ring. If $Gr(B)$ is semiprime, then $B$ is semiprime.
\end{proposition}
\begin{proof}
Let $I$ be a two-sided ideal of $B$ such that $I^2=0$. Then, $Gr(I)^2=0$ and hence $Gr(I)=0$. This
implies that $I=0$.
\end{proof}

\begin{theorem}[Goldie's theorem: bijective case]\label{2.3.8}
Let $R$ be a semiprime left Goldie ring and $A=\sigma(R)\langle x_1,\dots,x_n\rangle$ a bijective
skew \textit{PBW} extension of $R$. Then, $A$ is semiprime left Goldie, and hence, $Q_l(A)$ exists
and it is semisimple. The right side version of the theorem also holds.
\end{theorem}
\begin{proof}
By Goldie's theorem, $Q_l(R)=S_0(R)^{-1}R$ exists and it is semisimple. Note that for every $1\leq
i\leq n$, $\sigma_i(S_0(R))=S_0(R)$. By Lemma \ref{2.1.6}, $S_0(R)^{-1}A$ exists and it is a
bijective extension of $Q_l(R)$, i.e., $S_0(R)^{-1}A=\sigma(Q_l(R))\langle x_1',\dots,
x_n'\rangle$. Since $Q_l(R)$ is left Noetherian, then by Theorem \ref{1.3.4}, $S_0(R)^{-1}A$ is
left Noetherian, i.e, left Goldie. By Theorem \ref{1.3.2},
$Gr(S_0(R)^{-1}A)=Gr(\sigma(Q_l(R))\langle x_1',\dots, x_n'\rangle)$ is a quasi-commutative (and
bijective) extension of the semiprime left Goldie ring $Q_l(R)$, so by Theorem \ref{2.3.3},
$Gr(S_0(R)^{-1}A)$ is semiprime (left Goldie). Proposition \ref{2.3.5} says that $S_0(R)^{-1}A$ is
semiprime. In order to apply Proposition \ref{2.3.4} and conclude the proof only rest to observe
that $S_0(R)\subseteq S_0(A)$ and the left regular elements of $A$ coincide with $S_0(A)$. The last
statement can be justify in the following way: since $S_0(R)^{-1}A$ is semiprime left Goldie, then
the left regular elements of $S_0(R)^{-1}A$ coincide with its regular elements, so by Proposition
\ref{2.3.4a} the same is true for $A$.
\end{proof}

\section{The quantum version of the Gelfand-Kirillov conjecture for skew quantum polynomials}

As application of the results of the previous sections, we can prove a quantum version of the
Gelfand-Kirillov conjecture for the ring of skew quantum polynomials. This class of rings were
defined in \cite{lezamareyes1}, and represent a generalization of Artamonov's quantum polynomials
(see \cite{Artamonov},  \cite{Artamonov2}). They can be defined as a quasi-commutative bijective
skew $PBW$ extension of the $r$-multiparameter quantum torus, or also, as a localization of a
quasi-commutative bijective skew $PBW$ extension. We recall next its definition.

Let $R$ be a ring with a fixed matrix of parameters $\textbf{q}:=[q_{ij}]\in M_n(R)$, $n\geq 2$,
such that $q_{ii}=1=q_{ij}q_{ji}=q_{ji}q_{ij}$ for every $1\leq i,j\leq n$, and suppose also that
it is given a system $\sigma_1,\dots,\sigma_n$ of automorphisms of $R$. The ring of \textit{skew
quantum polynomials over $R$}, denoted by $R_{\textbf{q},\sigma}[x_1^{\pm 1 },\dots,x_r^{\pm 1},
x_{r+1},\dots,x_n]$, is defined as follows:
\begin{enumerate}
\item[\rm{(i)}]$R\subseteq R_{\textbf{q},\sigma}[x_1^{\pm 1},\dots,x_r^{\pm 1},
x_{r+1},\dots,x_n]$;
\item[\rm{(ii)}]$R_{\textbf{q},\sigma}[x_1^{\pm 1
},\dots,x_r^{\pm 1}, x_{r+1},\dots,x_n]$ is a free left $R$-module with basis
\begin{equation}\label{equ1.4.2}
\{x_1^{\alpha_1}\cdots x_n^{\alpha_n}|\alpha_i\in \mathbb{Z} \ \text{for}\ 1\leq i\leq r \
\text{and} \ \alpha_i\in \mathbb{N}\ \text{for}\ r+1\leq i\leq n\};
\end{equation}
\item[\rm{(iii)}] the variables $x_1,\dots,x_n$ satisfy the defining relations
\begin{center}
$x_ix_i^{-1}=1=x_i^{-1}x_i$, $1\leq i\leq r$,

$x_jx_i=q_{ij}x_ix_j$, $x_ir=\sigma_i(r)x_i$, $r\in R$, $1\leq i,j\leq n$.
\end{center}
\end{enumerate}
When all automorphisms are trivial, we write $R_{\textbf{q}}[x_1^{\pm 1 },\dots,x_r^{\pm 1},
x_{r+1},\dots,x_n]$, and this ring is called the ring of \textit{quantum polynomials over $R$}. If
$R=\Bbbk$ is a field, then $\Bbbk_{\textbf{q},\sigma}[x_1^{\pm 1 },\dots,x_r^{\pm 1},
x_{r+1},\dots,x_n]$ is the \textit{algebra of skew quantum polynomials}. For trivial automorphisms
we get the \textit{algebra of quantum polynomials} simply denoted by $\mathcal{O}_\textbf{q}$ (see
\cite{Artamonov}). When $r=0$, $R_{\textbf{q},\sigma}[x_1^{\pm 1},\dots,x_r^{\pm 1},
x_{r+1},\dots,x_n]=R_{\textbf{q},\sigma}[x_1,\dots,x_n]$ is the
\textit{$n$-mul\-ti\-pa\-ra\-me\-tric skew quantum space over $R$}, and when $r=n$, it coincides
with $R_{\textbf{q},\sigma}[x_1^{\pm 1},\dots,x_n^{\pm 1}]$, i.e., with the
\textit{$n$-multiparametric skew quantum torus over $R$}.

Note that $R_{\textbf{q},\sigma}[x_1^{\pm 1},\dots,x_r^{\pm 1}, x_{r+1},\dots,x_n]$ can be viewed
as a localization of the $n$-mul\-ti\-pa\-ra\-me\-tric skew quantum space, which, in turn, is an
skew $PBW$ extension. In fact, we have the quasi-commutative bijective skew $PBW$ extension
\begin{equation}\label{equ2.2.3}
A:=\sigma(R)\langle x_1,\dots, x_n\rangle, \, \text{with} \, x_ir=\sigma_i(r)x_i \, \text{and}\,
x_jx_i=q_{ij}x_ix_j, 1\leq i,j\leq n;
\end{equation}
observe that $A=R_{\textbf{q},\sigma}[x_1,\dots,x_n]$. If we set
\begin{center}
$S:=\{rx^{\alpha}\mid r\in R^*, x^{\alpha}\in {\rm Mon}\{x_1,\dots,x_r\}\}$,
\end{center}
then $S$ is a multiplicative subset of $A$ and
\begin{equation}\label{equ2.7.4}
S^{-1}A\cong R_{\textbf{q},\sigma}[x_1^{\pm 1},\dots,x_r^{\pm 1}, x_{r+1},\dots,x_n]\cong AS^{-1}.
\end{equation}

Before presenting our next result, let us first recall the classical Gelfand-Kirillov conjecture
and some well known cases, classical and quantum, where the conjecture have positive answer. We
start with the classical formulation.

\begin{enumerate}
\item[\rm (i)](Gelfand-Kirillov conjecture, \cite{GK}) \textit{Let $\mathcal{G}$ be an algebraic Lie algebra of finite dimension over a field $L$, with ${\rm
char}(L)=0$. Then, there exist integers $n,k\geq 1$ such that}
\begin{equation}\label{GKconjecture}
Q(\mathcal{U}(\mathcal{G}))\cong Q(A_n(L[s_1,\dots,s_k])),
\end{equation}
\textit{$\mathcal{U}(\mathcal{G})$ is the enveloping algebra of the Lie algebra $\mathcal{G}$ and
$A_n(L[s_1,\dots,s_k])$ is the Weyl algebra over the polynomial ring $L[s_1,\dots,s_k]$}.
\item[\rm (ii)] (\cite{GK}, Lemma 7) Let $\mathcal{G}$ be the algebra of all $n\times n$ matrices over a field $L$, i.e.,
$\mathcal{G}=M_n(L)$, with ${\rm char}(L)=0$. Then, $\mathcal{G}$ is algebraic and
$($\ref{GKconjecture}$)$ holds. The same is true if $\mathcal{G}$ is the algebra of matrices of
null trace.
\item[\rm (iii)] (\cite{GK}, Lemma
8) Let $\mathcal{G}$ be a finite dimensional nilpotent Lie algebra over a field $L$, with ${\rm
char}(L)=0$. Then, $\mathcal{G}$ is algebraic and $($\ref{GKconjecture}$)$ holds.
\item[\rm (iv)] (\cite{Joseph}, Theorem 3.2) Let $\mathcal{G}$ be a finite dimensional solvable algebraic Lie algebra over the field
$\mathbb{C}$ of complex numbers. Then, $\mathcal{G}$ satisfies the conjecture
$($\ref{GKconjecture}$)$.
\end{enumerate}

Now we review some well known results about the analog quantum version of the Gelfand-Kirillov
conjecture, where the Weyl algebra $A_n(L[s_1,\dots,s_k])$ in (\ref{GKconjecture}) is replaced by a
suitable $n$-multiparametric quantum space. $Z(B)$ will represent the center of the ring $B$.

\begin{enumerate}
\item[\rm (vi)](\cite{Alev2}, Theorem 2.15.) Let $U_q^{+}(sl_m)$ be the quantum enveloping algebra of the Lie algebra of strictly superior
triangular matrices of size $m\times m$ over a field $L$.
\begin{enumerate}
\item[\rm (a)]If $m=2n+1$, then
\begin{center}
$Q(U_q^{+}(sl_m))\cong Q({\rm K}_{{\rm q}}[x_1,\dots, x_{2n^2}])$,
\end{center}
where ${\rm K}:=Q(Z(U_q^{+}(sl_m)))$ and ${\rm q}:=[q_{ij}]\in M_{2n^2}(L)$, with
$q_{ii}=1=q_{ij}q_{ji}$ for every $1\leq i,j\leq 2n^2$.
\item[\rm (b)]If $m=2n$, then
\begin{center}
$Q(U_q^{+}(sl_m))\cong Q({\rm K}_{{\rm q}}[x_1,\dots, x_{2n(n-1)}])$,
\end{center}
where ${\rm K}:=Q(Z(U_q^{+}(sl_m)))$ and ${\rm q}:=[q_{ij}]\in M_{2n(n-1)}(L)$, with
$q_{ii}=1=q_{ij}q_{ji}$ for every $1\leq i,j\leq 2n(n-1)$.
\end{enumerate}
\item[\rm (vii)](\cite{Panov}, Main Theorem) Let $B$ be a pure $\textbf{\rm q}$-solvable $C$-algebra. Then, $Q(B)\cong Q(Gr(B))$ and $Gr(B)\cong
C_{\textbf{\rm q}}[x_1,\dots,x_n]$, where $C$ is a Noetherian commutative domain.
\item[\rm (viii)](\cite{Cauchon}) Let $L$ be a field and $B:=L[x_1][x_2;\sigma_2,\delta_2]\cdots[x_n;\sigma_n,\delta_n]$ an iterated
skew polynomial ring with some extra adequate conditions on $\sigma$'s and $\delta$'s. Then, there
exits $\textbf{\rm q}:=[q_{i,j}]\in M_n(L)$ with $q_{ii}=1=q_{ij}q_{ji}$, for every $1\leq i,j\leq
n$, such that $Q(B)\cong Q(L_{\rm q}[x_1,\dots,x_n])$.
\end{enumerate}

With the previous antecedents, our next result can be better understood.

\begin{corollary}[Gelfand-Kirillov conjecture for skew quantum polynomials]
Let $R$ be a left $($right$)$ Ore domain. Then,
\begin{center}
$Q(R_{\textbf{q},\sigma}[x_1^{\pm 1},\dots,x_r^{\pm 1}, x_{r+1},\dots,x_n])\cong
Q(\textbf{Q}_{\textbf{q},\sigma}[x_1,\dots,x_n])$,
\end{center}
where $\textbf{Q}:=Q(R)$.
\end{corollary}
\begin{proof}
In order to simplify the notation we write
$Q^{r,n}_{\textbf{q},\sigma}(R):=R_{\textbf{q},\sigma}[x_1^{\pm 1},\dots,x_r^{\pm 1},
x_{r+1},\dots,x_n]$. If $R$ is a domain, then $Q^{r,n}_{\textbf{q},\sigma}(R)$ is also a domain
(see Proposition \ref{1.1.10a}, (\ref{equ2.7.4}), and Proposition \ref{2.2.3}). Thus, from
Proposition \ref{2.2.3} and Theorem \ref{2.1.9} (or also using Theorem \ref{Orespbw}), if $R$ is a
left (right) Ore domain, then $Q^{r,n}_{\textbf{q},\sigma}(R)$ is a left (right) Ore domain, and
hence $Q^{r,n}_{\textbf{q},\sigma}(R)$ has left (right) total division ring of fractions,
$Q(Q^{r,n}_{\textbf{q},\sigma}(R))\cong Q(A)$, with $A$ as in (\ref{equ2.2.3}). Therefore, with the
notation of the previous sections, we have
\begin{center}
$Q(Q^{r,n}_{\textbf{q},\sigma}(R))\cong Q(A)\cong Q(\sigma(Q(R))\langle
x_1',\dots,x_n'\rangle)\cong Q(\textbf{Q}_{\textbf{q},\sigma}[x_1,\dots,x_n])$,
\end{center}
where $\textbf{Q}:=Q(R)$ and we identify $x_i'=\frac{x_i}{1}:=x_i$ and
$\overline{\sigma_i}:=\sigma_i$, $1\leq i\leq n$. Thus, we have proved that the left (right) total
rings of fractions of $Q^{r,n}_{\textbf{q},\sigma}(R)$ is the left (right) total ring of fractions
of the $n$-multiparametric skew quantum space over $Q(R)$.
\end{proof}

As another application of the results of the previous sections, we conclude the paper with the
Goldie's theorem for the skew quantum polynomials.

\begin{corollary}
Let $R$ be a semiprime left $($right$)$ Goldie ring, then $R_{\textbf{q},\sigma}[x_1^{\pm 1
},\dots,x_r^{\pm 1}, x_{r+1},\dots,x_n]$ is also a semiprime left $($right$)$ Goldie ring.
\end{corollary}
\begin{proof}
From Theorem \ref{2.3.3} (we can use also Theorem \ref{2.3.8}) we get that $A$ in (\ref{equ2.2.3})
is a semiprime left (right) Goldie ring. In addition, note that the set $S$ in (\ref{equ2.7.4})
satisfies the hypothesis of Proposition \ref{2.3.4}: in fact, since $A$ is semiprime left (right)
Goldie, any left (right) regular element is regular; $S\subseteq S_0(A)$ since if $rx^{\alpha}\in
S$ and $p=c_1x^{\beta_1}+\cdots+c_tx^{\beta_t}\in A$ are such that $rx^{\alpha}p=0$ or
$prx^{\alpha}=0$, then $p=0$ since $r$ and the constants $c_{\alpha,\beta}$ are invertible (Remark
\ref{identities}). Proposition \ref{2.3.4} says that $R_{\textbf{q},\sigma}[x_1^{\pm 1
},\dots,x_r^{\pm 1}, x_{r+1},\dots,x_n]$ is semiprime left (right) Goldie.
\end{proof}



\end{document}